\documentclass[a4paper]{amsart}

\usepackage[fontsize=12pt]{scrextend}
\usepackage{blindtext}

\usepackage[T2A,T1]{fontenc}
\usepackage[utf8]{inputenc}
\usepackage[russian,english]{babel}

\usepackage{amsmath}
\usepackage{amsthm}
\usepackage{amssymb}
\usepackage{amscd}
\usepackage[pdftex]{graphicx}
\usepackage[all]{xy}
\usepackage{euscript}

\usepackage[pdftex, colorlinks, citecolor=blue, linkcolor=black]{hyperref}
\usepackage[mathlines]{lineno}

\newcommand*\patchAmsMathEnvironmentForLineno[1]{%
  \expandafter\let\csname old#1\expandafter\endcsname\csname #1\endcsname
  \expandafter\let\csname oldend#1\expandafter\endcsname\csname end#1\endcsname
  \renewenvironment{#1}%
     {\linenomath\csname old#1\endcsname}%
     {\csname oldend#1\endcsname\endlinenomath}}%
\newcommand*\patchBothAmsMathEnvironmentsForLineno[1]{%
  \patchAmsMathEnvironmentForLineno{#1}%
  \patchAmsMathEnvironmentForLineno{#1*}}%
\AtBeginDocument{%
\patchBothAmsMathEnvironmentsForLineno{equation}%
\patchBothAmsMathEnvironmentsForLineno{multline}%
}

\theoremstyle{plain}
\newtheorem{theoremA}{Theorem}
\newtheorem{theorem}{Theorem}[section]
\newtheorem{lemma}[theorem]{Lemma}
\newtheorem{proposition}[theorem]{Proposition}

\theoremstyle{definition}
\newtheorem{remark}[theorem]{Remark}

\newtheorem{ex-constr}[theorem]{Construction}

\renewcommand{\dim}{\mathrm{dim}\,}

\renewcommand{\P}{{\mathbb P}}

\newcommand{\CH}{\mathrm{CH}}

\newcommand{\Pic}{\mathrm{Pic}}

\begin{document}

\author{Ivan Bazhov
}
\address{Institut de Math\'ematiques de Jussieu, 4 Place Jussieu, 75005, France}
\email{ibazhov@gmail.com}
\title{On the decomposition of the small diagonal of a K3 surface}
\date{}

\begin{abstract}
We give a new proof of the theorem of Beauville and Voisin about the 
decomposition of the small diagonal of a K3 surface $S$. Our proof is explicit and works with the embedding of $S$ in $\P^g$. It is different from the one used by Beauville and Voisin, which employed the existence of one-parameters families of elliptic curves.
\end{abstract}


\maketitle{}

The canonical zero cycle on a K3 surface $S$ is defined in \cite{BV} as the rational equivalence class of any point lying on a rational curve $C\subset S$. The paper \cite{BV} shows that the intersection of any two divisors in $S$ is proportional to the canonical cycle in $\CH_0(S)$. It is also shown that the second Chern class $c_2(S)$ is proportional to this canonical zero cycle $o$. Both results can be obtained as consequences of the following theorem.
\begin{theoremA}(\cite[Proposition 4.2]{BV})
\label{th_BV}
Let $S$ be a K3 surface. In $\CH_2(S^3)_{\mathbb Q}$ there is a decomposition
\begin{multline}
\label{decomposition}
\Delta_{123}=\Delta_{12}\times o_3+\Delta_{23}\times o_1+\Delta_{13}\times o_2\\
-S\times o\times o-o\times S\times o-o\times o\times S,
\end{multline}
where $o$ is any point representing the canonical zero cycle, 
 $\Delta_{123}$ is a small diagonal in $S^3$, and the notation $\Delta_{ij}\times o_k$ stands for $\pi_{ij}^*(\Delta)\cdot \pi_k^*o$.
\end{theoremA}

The goal of this paper is to give another proof of Theorem \ref{th_BV} for a K3 surface $S$ with $\Pic(S)=\mathbb Z[L]$ with $L^2=2g-2$.
Our proof is very explicit using the embedding of $S$ in $\P^g$. It is based on the study of the set of pairs of points $(x,y)$ in $S\times S$ such that two curves in the linear system $|L|$ intersect exactly at these two points with given multiplicities. Specifically, we choose the multiplicities $2g-3$ and $1$. In other words, we are studying the surface $\Sigma$ parameterising complete intersections subschemes of $S$ consisting in the union of two points, one of them with multiplicity $2g-3$. We will prove that this is a surface and will establish two relations (\ref{eq_1}) and (\ref{eq_2}), from which we obtain the relation (\ref{decomposition}) up to some multiplicative factor $\mu$, which is non-zero if the surface $\Sigma\subset S\times S$ dominates factors. The second part of the paper is then devoted to the proof that $\mu\neq 0$.
In order to prove this non-vanishing we will interpret the surface $\Sigma$ in a slightly different way: 
as $\Pic(S)=\mathbb Z[L]$, the curves in $|L|$ are irreducible hence the intersection of any two different curves in
the linear system $|L|$ is a zero-dimensional subscheme of $S$ of length $\mathrm{deg}(L)$, so we have a morphism $Gr(2,H^0(S,L))\to S^{[2g-2]}$ and we let $Gr$ denote the image.
Using techniques from \cite{CM, L98, LS00, V07} to work with cohomology groups
of the Hilbert scheme $S^{[2g-2]}$, one can define the pieces $E_M^*(Gr)\in \CH(S^{m})$ of the decomposition of the class of $Gr$ in $\CH(S^{[2g-2]})$, where $M$ is a partition of $\{1,\ldots, 2g-2\}$ and $m=|M|$. 
The proof that $\mu\neq0$ involves the study of this class $E_M^*(Gr)$ in the case where $M$ is a partition into two integers. 

{\bf Acknowledgement.} I am grateful to my advisor Claire Voisin for her kind help and guidance during the work. I am also grateful to Lie Fu for useful comments and discussions.

\section{The Proof}
\subsection{Surface $\Sigma$}
Let us recall that a K3 surface with a very ample linear system $L$ of degree $2g-2\geq 4$ generating $\Pic(S)$ can be embedded in $\P^g$ and  the intersection of $S$ with any linear subspace $\P^{g-2}\subset \P^g$ is a zero-cycle of degree $2g-2$ on $S$.
The set of all $P=\P^{g-2}\subset\P^g$ is the Grassmann variety $Gr(g-1,g+1)$ and as already mentioned this provides a morphism 
$$
Gr(g-1,g+1)\to S^{[2g-2]},
$$
which maps $[P]$ to $P\cap S$ for $[P]\in Gr(g-1,g+1)$. We denote by $Gr$ the image of this map. We introduce the incidence scheme $\Xi\subset S\times S^{[2g-2]}$:
\begin{equation}
\label{diag}
\xymatrix{
\Xi\ar[d]^{\pi_1}\ar[r]^{\pi_2}& S^{[2g-2]}\\
S&
}
\end{equation}
Let us put $\mathcal L^{[2g-2]}=(\pi_2)_*\pi_1^*(\mathcal O_S(L))$.
The image $Gr$ is the locus in $S^{[2g-2]}$, where the rank of the map
$$
H^0(S,H)\otimes \mathcal O_{S^{[2g-2]}}\to \mathcal L^{[2g-2]}
$$
is $g-1$ and it follows by \cite{fulton} that its class is given by 
\begin{equation}
\label{class_of_Gr}
c_{g-2}(\mathcal L^{[2g-2]})c_g(\mathcal L^{[2g-2]}))-c_{g-1}^2(\mathcal L^{[2g-2]}).
\end{equation}


Let us also define a subset $G'_0\subset Gr(g-1,g+1)$ in the following way
$$
G'_0=\{[P]\in Gr(g-1,g+1): P\cap S=(2g-3)p_1+p_2\, \mbox{as a cycle}\}.
$$
We also define the variety $\Sigma_0\subset S\times S$ as
\begin{multline*}
\Sigma_0:=\left\{(p_1,p_2)\in S\times S: (2g-3)p_1+p_2=S\cap P\right.\\\left. \mbox{for\, some}\, [P]\in Gr(g-1,g+1)\right\}.
\end{multline*}

If we construct the following diagram
\begin{equation}
\label{diagram2}
\xymatrix{
G'_0\subset E_{2g-3,1}\ar@<2ex>[d]^{p}\ar[r]^{\phantom{LL}q}& S^{[2g-2]}\\
\Sigma_0\subset S\times S\phantom{ll}&
}
\end{equation}
where $E_{2g-3,1}$ represents schemes of the form $(2g-3)p_1+p_2$ on $S$, we then have $G'_0=\pi_2^{-1}(Gr)$ and $\Sigma_0=\pi_1(\pi_2^{-1}(Gr))$.
Let us prove the following facts about geometry of $\Sigma_0$.

\begin{lemma}
\label{lemmaaboutsigma}
The following holds
\begin{enumerate}
\item $\Sigma_0$ is a non-empty surface, possibly reducible, 
\item there is a component $\Sigma\subset \Sigma_0$ such that $\Sigma$ dominates both factors of $S\times S$.
\end{enumerate}
\end{lemma}
\begin{proof}
Let us first proof that $\dim \Sigma_0\leq 2$. Indeed, from the equality
$$
p_1+(2g-3)p_2=L^2,
$$
valid in $\CH^2(S)$ for any pair $(p_1,p_2)\in\Sigma_0$, and the theorem of Mumford \cite{M}, it follows that 
\begin{equation}
\label{Kcc}
\pi_1^*\sigma_S+(2g-3)\pi_2^*\sigma_S=0
\end{equation}
on $\Sigma_0$, where $\sigma_S$ is a non-vanishing 2-form on $S$. Therefore $\dim \Sigma_0\leq2$.

Notice that equation (\ref{Kcc}) characterises K-correspondence in the terminology of \cite{V03}. This equation implies that for any irreducible component $\Sigma$ of $\Sigma_0$ the morphism $\pi_1|_{\Sigma}$ is dominant if and only if the morphism $\pi_2|_{\Sigma}$ is dominant. Indeed, these conditions are respectively equivalent to the generic non-vanishing of $\pi_i^*\sigma_S$.
This argument also shows that (a) and (b) are implied by the fact that the first projection $\pi_1|_{\Sigma_0}$ is dominant. 
In order to prove this last statement we observe that the cycle $\Omega=p_*q^*(Gr)$ has for support the surface $\Sigma_0$ although this cycle could be non-effective due to the fact that even if $\Sigma_0$ has the right dimension the scheme $G'_0=q^{-1}(Gr)$ could be of a higher dimension leading to excess  formulas in the computation of the cycle $\Omega$.
Nevertheless we can argue that if $\Omega$ can not be represented by a cycle supported on the union of divisors of the form $D\times S$, 
then one component of the support $\mathrm{supp}\,\Omega$ has to dominate $S$ by the first projection, that is, one component of $\Sigma_0$ dominates $S$ by the first projection.
The next section is devoted to the proof that the class $\Omega$ can not be supported on the union of $D\times S$, 
see Proposition \ref{thelemma}. 
\end{proof}

\begin{remark}
Let us note that we expect $\dim G'_0=2$, $\Sigma_0$ is irreducible and the projection $p:G'_0\to \Sigma_0$ is a one-to-one correspondence. In this case we get $[\Sigma_0]=\Omega$ in $H^4(S\times S)$ (actually, we can consider the equality even in $\CH^2(S\times S)$). Unfortunately, the author does not know how to prove these facts and we avoid them in our proof by introducing below a surface $G'$ as a substitute of $G'_0$.
\end{remark}

Let $\Sigma$ be a surface as in the last lemma and let $G'\subset G'_0$ be any surface dominating $\Sigma$ after the projection $\pi_1: G'_0\to \Sigma_0$. We can consider $G'$ as a subvariety of $Gr(g-1,g+1)$ and define $\Pi_2\subset \P^g\times \P^g$ and $\Pi_3\subset \P^g\times \P^g\times \P^g$ as the universal varieties:
$$
\Pi_2:=\left\{(p_1,p_2)\in \P^g\times \P^g: x_1,x_2\in P\, \mbox{for}\, [P]\in G'\right\},
$$
$$
\Pi_3:=\left\{(p_1,p_2,p_3)\in \P^g\times \P^g\times \P^g: x_1,x_2,x_3\in P\, \mbox{for}\, [P]\in G'\right\}.
$$
We have $\dim \Pi_2=\dim G'+2(g-2)=2g-2$ and $\dim \Pi_3=\dim G'+3(g-2)=3g-4$.
Clearly, $\Sigma\subset\Pi_2\cap S\times S$.

\subsection{Main result}
A key observation for our proof is the following lemma.
\begin{lemma}
\label{mu}
\begin{enumerate}
\item There is a decomposition in $\CH_{2}(S\times S)$ which, in fact, is an equality of effective cycles 
\begin{equation}
\label{eq_1}
\Pi_2|_{S\times S}=\alpha\Delta+\beta\left(\Sigma+\Sigma^T\right),
\end{equation}
where $\Delta$ is the diagonal in $S\times S$.
\item There is a decomposition in $\CH_{2}(S\times S\times S)$
\begin{multline}
\label{eq_2}
\Pi_3|_{S\times S\times S}=\\
\gamma\Delta_{123}+
\varepsilon\left(\delta_{12*}(\Sigma+\Sigma^T)+\delta_{23*}(\Sigma+\Sigma^T)+\delta_{31*}(\Sigma+\Sigma^T)\right),
\end{multline}
where $\Delta_{123}$ is the small diagonal in $S^3$, and $\delta_{12}(x,y)=(x,x,y)$, $\delta_{23}(x,y)=(y,x,x)$, $\delta_{31}(x,y)=(x,y,x)$.
\end{enumerate}
\end{lemma}

\begin{proof}
The proof of (\ref{eq_1}) follows from the facts that $\Pi_2|_{S\times S}$ is symmetric and supported on the union of the diagonal, $\Sigma$ and $\Sigma^T$, and that $\Sigma$ and $\Sigma^T$ are chosen to be irreducible. The proof of (\ref{eq_2}) is similar.
\end{proof}

\begin{lemma} We have:
\label{mu20}
\begin{enumerate}
\item the denominators of ratios $\frac{\alpha}{\beta}$ and $\frac{\gamma}{\varepsilon}$ are non-zero and both ratios are non-negative,
\item the following relation holds
\begin{equation}
\label{a_b}
\frac{\gamma}{\varepsilon}-\frac{3\alpha}{\beta}=-\left(\frac{\alpha}{\beta}+a+b\right),
\end{equation}
where $a$ and $b$ are the degrees of the projections of $\Sigma\subset S\times S$ to its factors.
\end{enumerate}
\end{lemma}
\begin{proof}
\begin{enumerate}
\item 
As numbers $\alpha, \beta, \gamma,\varepsilon$ are non-negative and we need only to show that $\beta,\varepsilon\neq 0$. Since the diagonal $\Delta$ can not be the restriction of a cycle from $\P^g\times \P^g$, we have $\beta\neq 0$, the proof of $\varepsilon\neq0$ is similar. (We use here the fact that $S$ has some transcendental cohomology, so that the cohomology class of the diagonal of $S$ does not vanish on a product $U\times U$, where $U\subset S$ is dense Zariski open.)

\item Projecting (\ref{eq_2}) to $S\times S$ and taking cohomology classes, we easily conclude that $\frac{\gamma}{\varepsilon}=\frac{2\alpha}{\beta}-a-b$, which is equivalent to 
(\ref{a_b}). 
\end{enumerate}
\end{proof}

\begin{proof}[Proof of Theorem \ref{th_BV}]
We chose a surface $\Sigma$ as in Lemma \ref{lemmaaboutsigma}. 
Due to \cite[Proposition 2.6]{BV}, $\delta_{ij*}(\Pi_2|_{S\times S})$ can be represented by a sum of $Z'|_{S^3}$ and $o_k\times\Delta_{ij}$. We also recall that $\delta_{ij*}(\Delta)=\Delta_{123}$. So, putting (\ref{eq_1}) and (\ref{eq_2}) together, we get a decomposition of the small diagonal:
\begin{equation}
\left(\frac{\gamma}{\varepsilon}-\frac{3\alpha}{\beta}\right)\Delta_{123}=
\alpha_1\Delta_{12}\times o_3+\alpha_2\Delta_{23}\times o_1+\alpha_3\Delta_{13}\times o_2
+Z|_{S^3},
\end{equation}
where $Z\subset \P^g\times\P^g\times\P^g$.

Projecting to $S\times S$ and taking the cohomology classes, we easily conclude that $\alpha_1=\alpha_2=\alpha_3=\gamma/\varepsilon-3{\alpha}/{\beta}$, and by previous lemma, $\alpha_1=-(\alpha/\beta+a+b)$. Since the decomposition of the small diagonal holds in cohomology (due to \cite{BV} and more generally \cite{OGrady}), we can deal with the term $Z|_{S^3}$ as follows: this term is a polynomial in $L_1, L_2, L_3$, where $L_i:=pr_i^*L$ and on the other hand it is cohomologous to
$$
-\alpha_1 (S\times o\times o+o\times S\times o+o\times o\times o\times S).
$$
By \cite{BV} it is thus rationally equivalent to $-\alpha_1 (S\times o\times o+o\times S\times o+o\times o\times o\times S)$. Since $a,b>0$ by choice of $\Sigma$ and $\alpha/\beta$ is non-negative, we can divide the equation by $-\alpha/\beta-a-b$ to get the result. The theorem is proved.
\end{proof}

\begin{remark}
We would like to emphasise that this proof is very  different from the one
used by Beauville and Voisin, which uses the existence of one-parameters families of elliptic curves. It is much more along the lines of the  method used by Voisin in  the Calabi--Yau hypersurface  case, and Fu in the Calabi-Yau complete intersection
 case (see \cite[Theorem 3.1]{V12}, \cite{Fu}). To study the case of Calabi-Yau varieties, we need to replace $G'$ by the set of lines intersecting the hypersurface in two points. In this case, the result \cite[Proposition 2.6]{BV} used in our proof, becomes \cite[Lemma 3.3]{V12}.

\end{remark}

\begin{remark} As proved in \cite{BV}, the decomposition of the small diagonal immediately
gives the fact that $c_2(S)$ is proportional to the canonical cycle $o$.
From our proof we can easily get  another more direct proof of this fact, using only (\ref{eq_1}). 
Indeed, let us intersect the decomposition (\ref{eq_1}) with $\Delta$. We get that $\alpha c_2(S)$ is a combination of a canonical zero cycle (corresponding to the term $\Pi_2|_{S\times S}$) and zero cycles supported on $\Sigma\cap\Delta$ and $\Sigma^T\cap\Delta$. But clearly the points on $\Sigma\cap\Delta$ are rationally equivalent to $o$. This proves the statement concerning 
$c_2$, once we prove that $\alpha\neq0$, which can be derived from the following remark or from the proofs of Lemma \ref{lemmaaboutsigma} and of Proposition \ref{thelemma} in the next section.
\end{remark}

\begin{remark}
\label{rem}
Let us present a relation between $\alpha/\beta$, $a$, and $b$. We see from the definition of $\Sigma$ that for any
$(x,y)\in\Sigma$, we have the equality
$$ (2g-3)x+y=L^2\,\,{\rm in}\,\,{\rm CH}_0(S).$$
It follows that  we have for any
$x\in S$
$$\Sigma_*(x)= -a(2g-3)x +C \,\,{\rm in}\,\,{\rm CH}_0(S),$$
where $C$ is a constant multiple of $L^2$.
Similarly
$$\Sigma^T_*(x)= -\frac{b}{2g-3}x+C' \,\,{\rm in}\,\,{\rm CH}_0(S).$$
Applying (\ref{eq_1}) and the fact that $\Pi_2$ is restricted from $\mathbb{P}^g\times \mathbb{P}^g$,
we thus conclude that for any $x\in S$
$$C''=-a(2g-3)x-\frac{b}{2g-3}x+\frac\alpha\beta x \,\,{\rm in}\,\,{\rm CH}_0(S),$$
where $C''$ is a constant multiple of $L^2$.
It follows that 
$$
\alpha/\beta=a(2g-3)+\frac{b}{2g-3}.
$$
\end{remark}

\section{Proof of the fact that  $\mathrm{supp}\,\Omega$ dominates factors}

The goal of this section is to prove the following lemma.
\begin{proposition}
\label{thelemma}
The class $\Omega=E^*_{2g-3,1}(Gr)$ in $H^*(S\times S)$ can not be represented by a cycle supported on the union of divisors of the form $D_i\times S$ and $S\times D_j$ and hence its support has non-trivial projections to factors of $S\times S$
\end{proposition}

To prove Proposition \ref{thelemma} we study $H^*(S^{[2g-2]})$ 
and introduce the following notation. Let
$$
M=(m_1,m_2,\ldots,m_k)
$$
be a partition of $\{1,\ldots, 2g-2\}$. Such a partition determines a partial diagonal
$$
S_M\cong S^k\subset S^{2g-2},
$$
defined by the conditions
$$
x=(x_1,\ldots, x_{2g-2})\in S_M\iff x_i=x_j\, \mbox{if}\,\, i,j\in m_l, \mbox{for some}\, l.
$$
Consider the quotient map
$$
q_M:S^k\cong S_M\to S^{(2g-2)},
$$
and denote by $E_M$ the following fibered product:
$$
E_M:= S_M\times_{S^{(2g-2)}}S^{[2g-2]}\subset S^k\times S^{[2g-2]}.
$$
We view $E_M$ as a correspondence between $S^k$ and $S^{[2g-2]}$ and we will denote by $E^*_M:\CH(S^{[2g-2]})\to\CH(S^m)$ the map
$$
\alpha\to \pi_{1*}(\pi_2^*(\alpha)\cdot E_M).
$$
The main point of the proofs is considering $E_{M}^*(Gr)$ for the partition $M=(\{1,\ldots,2g-3\},\{2g-2\})$ and the intersection
$$
\pi_{2*}(E_M)\cdot(c^2_{g-1}(\mathcal L^{[2g-2]})-c_g(\mathcal L^{[2g-2]})c_{g-2}(\mathcal L^{[2g-2]}))
$$
We now turn our attention to the cup product on Hilbert scheme $S^{[2g-2]}$.

\subsection{Cup product on $S^{[n]}$}

The paper \cite{LS00} gives a description on the ring structure on $H^*(S^{[n]})$; the following theorem holds (cf. \cite[Theorem 3.2]{LS00}):
\begin{theorem}
Let $S$ be a smooth projective surface with numerically trivial canonical class. Then there is a canonical isomorphism of graded rings
$$
(H^*(S;\mathbb Q)[2])^{[n]}\to H^*(S^{[n]};\mathbb Q)[2n].
$$
\end{theorem}

In the theorem above we define $A^{[n]}$ as
$$
A^{[n]}:=(A\{S_n\})^{S_n}.
$$
It is the subspace of invariants of the ring $A\{S_n\}$, which has the following grading by permutations in $S_n$
$$
A\{S_n\}:=\oplus_{\pi\in S_n}A^{\otimes (\pi)\textbackslash [n]}\cdot \pi.
$$

To describe $c_i(\mathcal O^{[2g-2]})$ in these terms, let us introduce the following notation. If $\sigma\in S_n$ is a permutation, then let $c(\sigma)$ be the number of cycles in $\sigma$ and $l(\sigma)=n-c(\sigma)$. The number $l(\sigma)$ is the minimal number of permutation needed to generate $\sigma$.

The statement \cite[Proposition 4.3]{LS00} (see also \cite{L98}) gives
\begin{equation}
\label{ci}
c_i(\mathcal O^{[2g-2]})=\epsilon_i,\ \mbox{where}\, \epsilon_i:=(-1)^i\sum_{l(\sigma)=i}\sigma\in H^*(S)^{[2g-2]}.
\end{equation}

The class of $E_{M}$ for $M=(\{1,\ldots,2g-3\},\{2g-2\})$
 is proportional to the sum of all permutations which contains one cycle of length $2g-3$.

\subsection{Two lemmas and the proof} Before we start the proof of Proposition \ref{thelemma}, we would like to state two lemmas about transpositions. Let us enumerate all transpositions in $S_{2g-3}$ by $s_1, s_2,\ldots s_{(2g-3)(g-2)}$ in such a way that
$$
l(s_1\cdot\ldots\cdot s_{2g-4})=2g-4
$$
and define $A(k)$ as the set of all permutations  $\sigma\in S_{2g-3}$ such that $l(\sigma s_i)>l(\sigma)$ for any $i\leq k$. We note that for $\sigma\in A(k)$ one has $\sigma\in A(k+1)$ or $\sigma=\sigma's_{k+1}$ for some $\sigma'\in A(k+1)$.

Let us define set of pairs $\mathcal F_k(i,j,\tau)$:
$$
\{(\sigma_1,\sigma_2)\in A_k\times A_k:l(\sigma_1)=i, l(\sigma_2)=j, l(\sigma_1\sigma_2\tau)=i+j+l(\tau)\}.
$$
And let $F_k(i,j,\tau)$ be the number of elements in $\mathcal F_k(i,j,\tau)$.

\begin{lemma}
\label{lp1}
If $i+1<j$, one has
\begin{equation}
\label{FF}
F_k(i,j,\tau)\leq F_k(i+1,j-1,\tau).
\end{equation}
and the inequality is strict in the case $i=g-3$, $j=g-1$, $k=0$, $\tau=id$.
\end{lemma}
\begin{proof}
The set $\mathcal F_k(i,j,\tau)$ can be divided in four subsets:
\begin{enumerate}
\item pairs $(\sigma_1,\sigma_2)$ such that $\sigma_1,\sigma_2\in A(k+1)$. This subset coincides with $\mathcal F_{k+1}(i,j,\tau)$.
\item pairs $(\sigma_1,\sigma_2)$ such that $\sigma_1\in A(k+1)$ and $\sigma_2\notin A(k+1)$, so $\sigma_2=\sigma'_2s_{k+1}$ with $l(\sigma_2')=j-1$. This subset is in bijection with pairs $(\sigma_1,\sigma_2')$ of $\mathcal F_{k+1}(i,j-1,s_{k+1}\tau)$.
\item pairs $(\sigma_1,\sigma_2)$ such that $\sigma_1\notin A(k+1)$ and $\sigma_2\in A(k+1)$, so $\sigma_1=\sigma'_1s_{k+1}$ with $l(\sigma_1')=i-1$. This subset is in bijection with pairs $(\sigma_1',s_{k+1}\sigma_2s_{k+1})$ of $\mathcal F_{k+1}(i-1,j,s_{k+1}\tau)$.
\item pairs $(\sigma_1,\sigma_2)$ such that $\sigma_1,\sigma_2\notin A(k+1)$, so $\sigma_1=\sigma_1's_{k+1}$ and $s_{k+1}\sigma_2'$, hence $l(\sigma_1\sigma_2\tau)\leq l(\sigma_1)\cdot l(\sigma_2)\cdot l(\tau)-2$. This subset is empty.
\end{enumerate}

So we have
\begin{multline}
F_k(i,j,\tau)=F_{k+1}(i,j,\tau)+F_{k+1}(i-1,j,s_{k+1}\tau)+\\
F_{k+1}(i,j-1,s_{k+1}\tau).
\end{multline}
The proof easily follows by induction if we prove that
$$
F_k(0,j,\tau)\leq F_k(1,j-1,\tau)
$$
for all $j, k$ and all $\tau$. Inequality follows from the fact that
any permutation $\sigma_2$, where $(id, \sigma_2)\in\mathcal F_k(0,j,\tau)$, has (more than one) decomposition $\sigma_1'\sigma_2'$, where $(\sigma_1',\sigma_2')\in\mathcal F_k(1,j-1,\tau)$. Different $\sigma_2$ has different decompositions.
We note that the inequality is strict if the lefthand side is positive, i.e., there exist at least one $\sigma_2$. We can consider
$$
(\sigma_1,\sigma_2)=(s_1\ldots s_{g-2}, s_{g-1}\ldots s_{2g-4})
$$
which provide with a non-emptyness of $F_{g-2}(0,g-2,\sigma_1^{-1})$.
\end{proof}

Let $\mathcal G(i,j)$ be the set of all pairs $(\sigma_1,\sigma_2)\in S_{2g-2}\times S_{2g-2}$ such that
$$
l(\sigma_1)+l(\sigma_2)-2=l(\sigma_1\sigma_2)
$$
and with the composition $\sigma_1\sigma_2$ preserves the point $2g-2$: $(\sigma_1\sigma_2)(2g-2)=2g-2$. Let $G(i,j)$ be the number of elements in $\mathcal G(i,j)$.

\begin{lemma}
\label{lp2}
One has
$$
G(g-1,g-1)>G(g-2,g).
$$
\end{lemma}
\begin{proof} Let us define a map of sets $f: S_n\to S_{n-1}$ by the following way: if $\sigma$ has a cycle $(\ldots, i,n,j,\ldots)$ we replace it by a cycles $(\ldots, i,j,\ldots)$. Clearly, if $\sigma(n)\neq n$ than $l(\sigma')=l(\sigma)-1$.

Now it is easy to see that $f(\mathcal G(i,j))$ is $(2g-3)-$fold covering of $\mathcal F_0(i-1,j-1,id)$, where $\mathcal F_0(i-1,j-1,id)$ was defined before Lemma \ref{lp1}. Due to Lemma \ref{lp1}, we get the result.
\end{proof}

\begin{proof}[Proof of Proposition \ref{thelemma}]


We recall that class of $Gr$ is given by (\ref{class_of_Gr}).
Using Grothendieck-Riemann-Roch theorem it is easy to see that classes $c_k(\mathcal L^{[2g-2]})$ are polynomials in $L_i$ and in other classes. We will be interesting in the coefficient of the diagonal in the decomposition of $E^*_{2g-3,1}(Gr)$, so we pay attention only to the part without $L_i$. Now we can put formally $L=0$ and we consider the class
$$
c_{g-1}^2(\mathcal O^{[2g-2]})-c_{g-2}(\mathcal O^{[2g-2]})c_{g}(\mathcal O^{[2g-2]}).
$$
The classes $c_i(\mathcal O^{[2g-2]})$ are given by (\ref{ci}) and we need to understand the sum
$$
\sum_{M, l(\sigma_1)=g-1,l(\sigma_2)=g-1}\sigma_M\sigma_1 \sigma_2
-
\sum_{M, l(\sigma_1)=g-2,l(\sigma_2)=g}\sigma_M\sigma_1 \sigma_2
$$
Every summand correspond to a class of $c_2(S)\times S$ (or $S\times c_2(S)$) or $\Delta$ in $S\times S$.
To distinguish the classes of diagonals, we need to require the factor $e^{g(\sigma_M,\sigma_1,\sigma_2)}$, which appears in triple intersection, is equal to one (cf. \cite[Proof of Lemma 2.13]{LS00}). It implies that three following conditions hold
\begin{enumerate}
\item $l(\sigma_1\sigma_2)=l(\sigma_M)$, in particular, $l(\sigma_1\sigma_2)=l(\sigma_1)+l(\sigma_2)-2$,
\item there is an element $i\in \{1,\ldots,2g-2\}$ such that $(\sigma_1\sigma_2)(i)=i$,
\item there is no element $i\in\{1,\ldots,2g-2\}$ such that $i=\sigma_1(i)=\sigma_2(i)$.
\end{enumerate}
Actually, the first two condition follows from the requirement $\sigma_M\sigma_1\sigma_2=id$.

The pairs $(\sigma_1,\sigma_2)$ with $l(\sigma_1)=i$, $l(\sigma_2)=j$ satisfying these requirements are precisely the set $\mathcal G(i,j)$ defined previously. Since $G(g-1,g-1)>G(g-2,g)$ by Lemma \ref{lp2}, we get that the class $\Omega$ contains diagonal. Therefore $\Omega$ viewed as a self-correspondence of $S$ does not act trivially on $H^{2,0}(S)$ and can not be supported on divisors of the form $D_i\times S$ and $S\times D_j$.
\end{proof}

\bibliographystyle{abbrv}

\end{document}